\documentclass{IEEEtran} 

\usepackage{epsfig}
\ifCLASSINFOpdf
\else
\fi
\usepackage[cmex10]{amsmath}

\graphicspath{{fig/}}

\usepackage{amsfonts}
\usepackage{amssymb}

\usepackage{ifthen}

\newtheorem{nnassumption}{\bf Assumption}

\newtheorem{nntheorem}{\bf Theorem}
\newenvironment{theorem}{\begin{nntheorem}\it}{\end{nntheorem}}
\newtheorem{nncorollary}{\bf Corollary}
\newenvironment{corollary}{\begin{nncorollary}\it}{\end{nncorollary}}
\newtheorem{nndefinition}{\bf Definition}

\newtheorem{nnproposition}{\bf Proposition}

\newtheorem{nnproblem}{\bf Problem}

\newtheorem{nnlemma}{\bf Lemma}
\newenvironment{lemma}{\begin{nnlemma}\it}{\end{nnlemma}}
\newtheorem{nnremark}{\bf Remark}
\newenvironment{remark}{\begin{nnremark} \rm }{\hfill \hspace*{1pt}\hfill $\circ$\end{nnremark}}
\newtheorem{nnexample}{\bf Example}

\newenvironment{proof}{{\bf Proof.}}{\hfill \hspace*{1pt}\hfill $\Box$}

\usepackage{color}

\def\R{\mathrm{I\kern-0.21emR}}
\def\N{\mathrm{I\kern-0.21emN}}

\renewcommand{\geq}{\geqslant}
\renewcommand{\leq}{\leqslant}

\begin{document}
\title{Feedback stabilization of a 1D linear reaction-diffusion equation with delay boundary control}

\author{Christophe Prieur
and
Emmanuel Tr\'elat
\thanks{C. Prieur is with Univ. Grenoble Alpes, CNRS, Gipsa-lab, F-38000 Grenoble, France, 
\texttt{christophe.prieur@gipsa-lab.fr}}
\thanks{E. Tr\'elat is with Sorbonne Universit\'es, UPMC Univ Paris 06, CNRS UMR 7598, Laboratoire Jacques-Louis Lions, F-75005, Paris, France, \texttt{emmanuel.trelat@upmc.fr}}
}

\maketitle

\begin{abstract}
The goal of this work is to compute a boundary control of reaction-diffusion partial differential equation. The boundary control is subject to a constant delay, whereas the equation may be unstable without any control. For this system equivalent to a parabolic equation coupled with a transport equation, a prediction-based control is explicitly computed. To do that we decompose the infinite-dimensional system into two parts: one finite-dimensional unstable part, and one stable infinite-dimensional part. An finite-dimensional delay controller is computed for the unstable part, and it is shown that this controller succeeds in stabilizing the whole partial differential equation. The proof is based on a an explicit form of the classical Artstein transformation, and an appropriate Lyapunov function.  A numerical simulation illustrate the constructive design method.
\end{abstract}


\IEEEpeerreviewmaketitle

\section{Introduction and main result}\label{sec_intro}

\subsection{Literature review and statement of the main result}

There have been a number of works in the literature dealing with the stabilization of processes with input delays, mainly in finite dimension, but seemingly much less for processes driven by PDEs.

In \cite{FridmanOrlov2009} a stable PDE is controlled by means of a delayed bounded linear control operator (see also \cite{solomon2015stability} for a semilinear case). 
In the present work, the control operator is unbounded (Dirichlet boundary control) and the open-loop system is unstable.

Unbounded control operators have been considered in \cite{nicaise:dcds:2009,Nicaise2007425,nicaise2008stabilization} for both wave and heat equations, where time-varying delays are allowed with a bound on the time-derivative of the delay function. See also \cite{fridman:SIAM:2010} for a second-order evolution equation. In the present paper, a Lyapunov technique is developed in which, in addition to an exponential stability analysis, we also design a stabilizing controller which is of a finite-dimensional nature, based on a finite-dimensional spectral truncation of \eqref{eqcont0} containing all unstable modes.

To the best of our knowledge, the first work dealing with input delayed unstable PDEs is \cite{Krstic:SCL:2009} where a reaction-diffusion equation is considered, and a backstepping approach is developed to stabilize it (see also \cite{bresch2014output} for a similar approach for a wave equation). In this paper, we do not use backstepping and we exploit a decomposition of the state space into a stable part and a finite-dimensional unstable part.

Let us write more precisely the problem under study and state the main result of this work. Let $L>0$, let $c\in L^\infty(0,L)$ and let $D\geq 0$ be arbitrary.
We consider the one-dimensional reaction-diffusion equation on $(0,L)$ with a delayed Dirichlet boundary control
\begin{equation} \label{eqcont0}
\begin{split}
&y_t = y_{xx}+c(x)y, \ t\geq 0, \; x \in (0,L), \\
&y(t,0)=0,\quad y(t,L)=u_D(t)=u(t-D),\ t\geq 0,
\end{split}
\end{equation}
where the state is $y(t,\cdot):[0,L]\rightarrow\R$ and the control is $u_D(t) = u(t-D)\in\R$, with $D> 0$ a constant delay.
Our objective is to design an exponentially stabilizing feedback control for \eqref{eqcont0}.

By using a classical change of variables (see e.g. \cite{krstic2008compensating}) this problem is equivalent to the problem of stabilizing the coupled system
\begin{equation*}
\begin{split}
&y_t = y_{xx}+c(x)y,   \ t\geq 0, \; x \in (0,L), \\
&y(t,0)=0,\quad y(t,L)=z(t,0),\ t\geq 0,\\
&z_t = z_w, \ t\geq 0, \ t\geq 0,\; w \in (0,D), \\
&z(t,D) = U(t), \ t\geq 0, 
\end{split}
\end{equation*}
with $z(t,w)=U(t+w-D)$, where the first equation is \eqref{eqcont0} and the second equation is a transport equation causing the delay $D$ in the control of \eqref{eqcont0}. In other words our control objective can be seen as a boundary control problem of a coupled system obtained by writing in cascade a reaction-diffusion equation and a transport equation.

We assume that we are only interested in what happens for $t\geq 0$, and we consider an initial condition
$$
y(0,\cdot)=y_0(\cdot) \in L^2(0,L),
$$
and since the boundary control is retarded with the delay $D$, we assume that no control is applied (i.e., $u=0$) within the time interval $(0,D)$. 
For every $t>D$ on, a nontrivial control 
can then be applied. 

In this paper, we establish the following result.

\begin{theorem}\label{thm1}
The delayed Dirichlet boundary control reaction-diffusion equation \eqref{eqcont0} is exponentially stabilizable, with a feedback control that is built from a finite-dimensional autonomous linear control system with input delay. When closing the loop with this feedback, the PDE \eqref{eqcont0} is exponentially stable, that is there exist $\mu>0$ and $C>0$ such that, for all $y_0(\cdot) \in L^2(0,L)$, the solution of \eqref{eqcont0} is such that $t\mapsto \|y(t,\cdot)\|_{H^1(0,L)}$ converges exponentially to $0$ as $t\rightarrow+\infty$.
\end{theorem}

Note that, in the previous result, we do not make any smallness assumption on the delay $D$: for any value $D\geq 0$ of the delay, there exists a stabilizing feedback.

\subsection{Presentation of the design method and organization of the paper}

The delayed controller considered in TheoremÊ\ref{thm1} is built and our approach yields a constructive design method. More precisely, our strategy, developed in Section \ref{sec2}, begins with a spectral analysis of the operator underlying the control system \eqref{eqcont0} (compact perturbation of a Dirichlet-Laplacian), thanks to which we split the system into two parts. The first part of the system is finite dimensional and contains (at least) all unstable modes, whereas the second part is infinite dimensional and contains only stable modes. The stabilizing feedback is designed on the finite-dimensional part of the system: we use the Artstein model reduction and we design a Kalman gain matrix in a standard way with the pole-shifting theorem; then, following \cite{breschprieurtrelat} we invert the Artstein transform and we obtain the desired feedback. This feedback control is such that its value $u(t-D)$ at time $t-D$ only depends on the values of $X_1(s)$ with $0<s<t-D$, where $X_1$ is identified with the unstable finite-dimensional part of the state.

By definition, this feedback stabilizes exponentially the finite-dimensional part of the system. Using an appropriate Lyapunov function, we then prove that it stabilizes as well the whole system. This is the core of the proof of Theorem \ref{thm1}.

The idea of designing a feedback on the unstable part of the system can be found in \cite{russell1978controllability} and has been used for instance in \cite{coron-trelat-heat,coron-trelat-wave} (for undelayed PDEs) where the efficiency of such a procedure has also been shown. Here, due to the presence of a delay, in practice one has to stabilize a finite-dimensional autonomous linear control system with input delay. In the existing literature, this classical issue has been investigated for instance in \cite{manitius1979finite,Bresch2014:IEEETAC} by a predictor approach. The recent paper \cite{breschprieurtrelat} surveys on the numerical and practical aspects of this problem and shows that the designed controller can be computed numerically in particular thanks to a fixed point procedure. Here, we exploit this procedure to design a stabilizing controller for the unstable heat equation, by revisiting the delay input for the unstable finite-dimensional part of the state space, and by adapting it to the full boundary delay control.

Overall, our stabilization procedure is carried out with a simple approach that is easy to implement. Some details and a numerical illustration is provided in Section \ref{secnum}.

The remaining part of the paper is organized as follows. Section \ref{sec2} is devoted to the proof of the main result and to the design of the delayed boundary controller. To do that we first decouple the reaction-diffusion equation into two coupled parts: one unstable finite dimensional part and one infinite-dimensional part, using a spectral decomposition. It allows us to explicitly compute a finite dimensional delay controller in Section \ref{finite:control:sec}. When closing the loop with this delay input, we prove that the PDE (\ref{eqcont0}) is exponentially stable by using an appropriate Lyapunov function. A numerical simulation is given in Section \ref{secnum}, highlighting the applicability of this design method. Section \ref{sec_proof_lemtechnique} contains the proof of an intermediate result. Finally Section \ref{sec:con} collects concluding remarks and points out possible research lines.

\section{Construction of the feedback and proof of Theorem \ref{thm1}}\label{sec2}
\subsection{Spectral reduction}
First of all, in order to deal rather with a homogeneous Dirichlet problem (which is more convenient), we set
\begin{equation}\label{defw}
w(t,x)=y(t,x)-\frac{x}{L}u_D(t),
\end{equation}
and we suppose that the control $u_D$ is differentiable for all positive times (this will be true in the construction that we will carry out). This leads to
\begin{equation}\label{reducedproblem2}
\begin{split}
& w_t=w_{xx}+cw+\frac{x}{L}cu_D-\frac{x}{L}u_D' ,  \quad \forall t>0, \;\forall x \in (0,1) , \\
& w(t,0)=w(t,L)=0, \quad \forall t>0, \\
& w(0,x)=y(0,x)-\frac{x}{L}u_D(0) , \quad \forall x \in (0,1).
\end{split}
\end{equation}
We define the operator
\begin{equation}
A=\partial_{xx}+c(\cdot)\mathrm{id},\ 
\end{equation}
on the domain $D(A)=H^2(0,L)\cap H^1_0(0,L)$. 
Then the above control system is
\begin{equation}\label{newzero}
w_t(t,\cdot)=Aw(t,\cdot)+a(\cdot)u_D(t)+b(\cdot)u_D'(t),
\end{equation}
with $a(x)=\frac{x}{L}c(x)$ and $b(x)=-\frac{x}{L}$ for every $x\in(0,L)$.

Noting that $A$ is selfadjoint and of compact inverse, we consider a Hilbert basis $(e_j)_{j\geq 1}$ of $L^2(0,L)$ consisting of eigenfunctions of $A$, associated with the sequence of eigenvalues $(\lambda_j)_{j\geq 1}$. Note that
$$-\infty<\cdots<\lambda_j<\cdots<\lambda_1\quad \textrm{and}\quad \lambda_j\underset{j\rightarrow +\infty}{\longrightarrow}-\infty ,$$
and that $e_j(\cdot)\in H^1_0(0,L)\cap C^2([0,L])$ for every $j\geq 1$.
Every solution $w(t,\cdot)\in H^2(0,L)\cap H^1_0(0,L)$ of \eqref{newzero} can be expanded as a series in the eigenfunctions $e_j(\cdot)$, convergent in $H_0^1(0,L)$,
$$w(t,\cdot)=\sum_{j=1}^{\infty}w_j(t)e_j(\cdot),$$
and therefore \eqref{eqcont0} is equivalent to the infinite-dimensional control system
\begin{equation}\label{eq16}
w_j'(t)=\lambda_jw_j(t)+a_ju_D(t)+b_ju_D'(t),\qquad\forall j\in\N^*,
\end{equation}
with
\begin{equation}
\begin{split}
a_j&=\left\langle a(\cdot),e_j(\cdot)\right\rangle_{L^2(0,L)} = \frac{1}{L}\int_0^L xc(x)e_j(x)\, dx, \\
b_j&=\langle b(\cdot),e_j(\cdot)\rangle_{L^2(0,L)} = -\frac{1}{L}\int_0^L xe_j(x)\, dx,
\end{split}
\end{equation}
for every $j\in\N^*$.
We define
\begin{equation}\label{eqalphaD}
\alpha_D(t)=u_D'(t),
\end{equation}
and we consider from now on $u_D(t)$ as a state and $\alpha_D(t)$ as a control (destinated to be a delayed feedback, with constant delay $D$), so that equations \eqref{eq16} and \eqref{eqalphaD} form an infinite-dimensional control system controlled by $\alpha_D$, written as
\begin{equation}\label{sys-dim-infinie}
\begin{split}
u_D'(t) &= \alpha_D(t), \\
w_1'(t) &=\lambda_1w_1(t)+a_1u_D(t)+b_j\alpha_D(t), \\
& \vdots \\
w_j'(t) &=\lambda_jw_j(t)+a_ju_D(t)+b_j\alpha_D(t), \\
& \vdots 
\end{split}
\end{equation}
and which is equivalent to \eqref{eqcont0}.

Let $n\in\N^*$ be the number of nonnegative eigenvalues and let $\eta>0$ be such that
\begin{equation}\label{refeta}
\forall k>n\quad \lambda_k<-\eta<0.
\end{equation}
Let $\pi_1$ be the orthogonal projection onto the subspace of $L^2(0,L)$ spanned by
$e_1(\cdot),\ldots,e_n(\cdot)$, and let
\begin{equation}\label{newun}
w^1(t)=\pi_1w(t,\cdot)=\sum_{j=1}^nw_j(t)e_j(\cdot).
\end{equation}
With the matrix notations
\begin{eqnarray}
\nonumber
&X_1(t)=\begin{pmatrix} u_D(t) \\ w_1(t) \\ \vdots \\ w_n(t)
\end{pmatrix} ,& \!\!\!\!
A_1=\begin{pmatrix}
0         &       0         & \cdots &    0           \\
a_1 & \lambda_1 & \cdots &    0           \\
\vdots    &  \vdots         & \ddots &   \vdots       \\
a_n &  0              & \cdots & \lambda_n
\end{pmatrix} , \\
\label{eq:def:A1}
&B_1=\begin{pmatrix} 1 & b_1 & \ldots &
b_n \end{pmatrix}^\top , &
\end{eqnarray}
the $n$ first equations of 
\eqref{sys-dim-infinie}
form the finite-dimensional control system with input delay
\begin{equation}\label{systfini}
X'_1(t)=A_1X_1(t) + B_1\alpha_D(t) = A_1X_1(t) + B_1\alpha(t-D).
\end{equation}
Note that the state $X_1(t)\in\R^{n+1}$ involves the term $u_D(t)$ which contains the delay. 

Our objective is to design a feedback control $\alpha$ exponentially stabilizing the infinite-dimensional system \eqref{sys-dim-infinie}. 
As shortly explained in the previous section, the idea consists of first designing a feedback control exponentially stabilizing the finite-dimensional system \eqref{systfini}, and then of proving that this feedback actually stabilizes the whole system \eqref{sys-dim-infinie}. The idea underneath is that the finite-dimensional system \eqref{systfini} contains all unstable modes of the complete system \eqref{sys-dim-infinie}, and thus has to be stabilized. It is however not obvious that this feedback stabilizing the unstable finite-dimensional part actually stabilizes as well the entire system \eqref{sys-dim-infinie}. This fact will be proved thanks to an appropriate Lyapunov functional.

\subsection{Stabilization of the unstable finite-dimensional part}\label{finite:control:sec}
Let us design a feedback control stabilizing the finite-dimensional linear autonomous control system with input delay \eqref{systfini} and let us also design a Lyapunov functional.
First of all, following the so-called Artstein model reduction (see \cite{artstein1982linear,richard:survey}), we set, for every $t\in\R$,
\begin{equation}\label{transfo_Artstein}
\begin{array}{rcl}Z_1(t)& =& X_1(t) + \int_{t-D}^t e^{(t-s-D)A_1}B_1\alpha(s)\, ds \\
&= &X_1(t) + \int_{0}^D e^{-\tau A_1}B_1\alpha(t-D+\tau)\, d\tau  ,
\end{array}\end{equation}
and we get that \eqref{systfini} is equivalent to
\begin{equation}\label{systfiniArtstein}
\dot{Z}_1(t) = A_1Z_1(t)+e^{-DA_1}B_1\alpha(t),
\end{equation}
which is a usual linear autonomous control system without input delay in $\R^{n+1}$.
The equivalence is because the Artstein transformation \eqref{transfo_Artstein} can be inverted (see further).
Now, for this classical finite-dimensional system, we have the following result.

\begin{lemma}\label{lem:1}
For every $D\geq 0$, the pair $(A_1,e^{-DA_1}B_1)$ satisfies the Kalman condition, that is,
\begin{equation}\label{kalman}
\mathrm{rank} \left( e^{-DA_1}B_1, A_1e^{-DA_1}B_1, \ldots, A_1^{n}e^{-DA_1}B_1  \right) = n+1 .
\end{equation}
\end{lemma}

\begin{proof}
Since $A_1$ and $e^{-DA_1}$ commute, and since $e^{-DA_1}$ is invertible, we have
\begin{equation*}
\begin{split}
& \mathrm{rank} \left( e^{-DA_1}B_1, A_1e^{-DA_1}B_1, \ldots, A_1^{n}e^{-DA_1}B_1  \right) \\
=&\ \mathrm{rank} \left( e^{-DA_1}B_1, e^{-DA_1}A_1B_1, \ldots, e^{-DA_1}A_1^{n}B_1  \right) \\
=&\ \mathrm{rank} \left( B_1, A_1B_1, \ldots, A_1^{n}B_1  \right),
\end{split}
\end{equation*}
and hence it suffices to prove that the pair $(A_1,B_1)$ satisfies the Kalman condition.
A simple computation leads to
\begin{equation}\label{deter1}
\textrm{det} \left( B_1, A_1B_1, \ldots, A_1^{n}B_1 \right)
= \prod_{j=1}^{n}(a_j+\lambda_jb_j)\,\mathrm{VdM}(\lambda_1,\ldots,\lambda_n) ,
\end{equation}
where $\mathrm{VdM}(\lambda_1,\ldots,\lambda_n)$ is a Van der Monde
determinant, and thus is never equal to zero since the real numbers $\lambda_j$, $j=1\ldots n$, are all distinct.
On the other part, using the fact that every $e_j(\cdot)$ is an eigenfunction of $A$ and belongs to $H^1_0(0,L)$, we have, for every integer $j$,
\begin{equation*}
\begin{array}{rcl}
a_j+\lambda_jb_j&
=& \frac{1}{L} \int_0^L x\left( c(x)e_j(x) -\lambda_je_j(x) \right) dx  
\\
&=& -\frac{1}{L} \int_0^L x e_j''(x) \, dx  
= - e_j'(L), \end{array}
\end{equation*}
which is not equal to zero since $e_j(L)=0$ and $e_j(\cdot)$ is a nontrivial solution of a linear second-order scalar differential equation. The lemma is proved.
\end{proof}

Since the linear control system \eqref{systfiniArtstein} satisfies the Kalman condition, the well-known pole-shifting theorem imply the existence of a stabilizing gain matrix and of a Lyapunov functional (see, e.g., \cite{Khalil:2002,Sontag:book,Trelat2005}). This yields the following corollary.

\begin{corollary}\label{corkalman}
For every $D\geq 0$, there exists a $1\times (n+1)$ matrix $K_1(D)=\begin{pmatrix} k_0(D),k_1(D),\ldots,k_n(D) \end{pmatrix} $ such that $A_1+B_1e^{-DA_1}K_1(D)$ admits $-1$ as an eigenvalue with order $n+1$.
Moreover there exists a $(n+1)\times (n+1)$ symmetric positive definite matrix $P(D)$ such that
\begin{equation}\label{poleshifting}\begin{array}{c}
P(D)\left(A_1+B_1e^{-DA_1}K_1(D)\right)\\+ \left(A_1+e^{-DA_1}B_1K_1(D)\right)^\top P(D) = -I_{n+1} .
\end{array}\end{equation}
In particular, the function
\begin{equation}\label{defV1}
V_1(Z_1) = \frac{1}{2} Z_1^\top P(D)Z_1
\end{equation}
is a Lyapunov function for the closed-loop system 
$$
\dot{Z}_1(t) = (A_1+e^{-DA_1}B_1K_1(D))Z_1(t) .
$$
\end{corollary}

\begin{remark}
It is even possible to choose $K_1(D)$ and $P(D)$ as smooth (i.e., of class $C^\infty$) functions of $D$, but we do not need this property in this paper.
\end{remark}

\begin{remark} 
In the statement above, we chose $-1$ as an eigenvalue of $A_1+B_1e^{-DA_1}K_1(D)$, but actually the pole-shifting theorem implies that, for every $(n+1)$-tuple $(\mu_0,\ldots,\mu_n)$ of eigenvalues there exists a $1\times (n+1)$ matrix $K_1(D)$ such that the eigenvalues $A_1+B_1e^{-DA_1}K_1(D)$ are exactly $(\mu_0,\ldots,\mu_n)$. The eigenvalue $-1$ was chosen here only for simplicity. What is important is to ensure that $A_1+B_1e^{-DA_1}K_1(D)$ is a \textit{Hurwitz} matrix (i.e., a matrix of which all eigenvalues have negative real part).

In practice, other choices can be done, which can be more efficient according to such or such criterion. For instance, instead of using the pole-shifting theorem, one could design a stabilizing gain matrix $K_1$ by using a standard Riccati procedure.
\end{remark}

\begin{remark}\label{rem_Lyap1}
From Corollary \ref{corkalman}, we infer that, for every $D\geq 0$, there exists $C_1(D)>0$ (depending smoothly on $D$) such that
\begin{equation}
\frac{d}{dt}V_1(Z_1(t)) = -\Vert Z_1(t)\Vert_{\R^{n+1}}^2 \leq -C_1(D)\, V_1(Z_1(t)),
\end{equation}
where $\Vert\ \Vert_{\R^{n+1}}$ is the usual Euclidean norm in $\R^{n+1}$.
\end{remark}

From Corollary \ref{corkalman}, the feedback $\alpha(t)=K_1(D)Z_1(t)$ stabilizes exponentially the control system \eqref{systfiniArtstein}.
Since $\alpha(t-D)$ is used in the control system \eqref{systfini}, and since in general we are only concerned with prescribing the future of a system, starting at time $0$, we assume that the control system \eqref{systfini} is uncontrolled for $t<0$, and from the starting time $t=0$ on we let the feedback act on the system. In other words, we set
\begin{equation}\label{defalpha}
\alpha(t) = \left\{ \begin{array}{ll}
0 & \textrm{if}\ t<D,\\
K_1(D)Z_1(t) & \textrm{if}\ t\geq D,
\end{array}\right.
\end{equation}
so that, with this control, the control system \eqref{systfini} with input delay is written as
$$
X_1'(t) = A_1X_1(t) + \chi_{(D,+\infty)}(t)B_1 K_1(D) Z_1(t-D),
$$
with $Z_1$ given by \eqref{transfo_Artstein}.
Here the notation $\chi_E$ stands for the characteristic function of $E$, that is the function defined by $\chi_{E}(t)=1$ whenever $t\in E$ and $\chi_{E}(t)=0$ otherwise.
Using \eqref{transfo_Artstein}, the feedback $\alpha$ defined by \eqref{defalpha} is such that, for all $t<D$,
\begin{subequations}\label{def_alpha}
\begin{equation}
\alpha(t) = 
0
\end{equation}
and, for all $t\geq D$,
\begin{equation}
\begin{array}{rcl}\alpha(t) &=&K_1(D) X_1(t) \\
&&+ K_1(D) \int_{\max(t-D,D)}^t e^{(t-D-s)A_1}B_1\alpha(s)\, ds  .\end{array}\end{equation}\end{subequations}
In other words, the value of the feedback control $\alpha$ at time $t$ depends on $X_1(t)$ and of the controls applied in the past (more precisely, of the values of $\alpha$ over the time interval $(\max(t-D,D),t)$).

\begin{lemma}
When closing the loop with the feedback \eqref{def_alpha}, the control system \eqref{systfini} is exponentially stable.
\end{lemma}

\begin{proof}
By construction $t\mapsto Z_1(t)$ converges exponentially to $0$, and hence $t\mapsto \alpha(t)$ and thus $t\mapsto \int_{\max(t-D,D)}^t e^{(t-D-s)A_1)}B_1\alpha(s)\, ds$ converges exponentially to $0$ as well. Then the equality \eqref{transfo_Artstein} implies that $t\mapsto X_1(t)$ converges exponentially to $0$.
\end{proof}

\noindent\textbf{Inversion of the Artstein transform.}
We are going to invert the Artstein transform, with two motivations in mind:
\begin{itemize}
\item First of all, it is interesting to express the stabilizing control $\alpha$ (defined by \eqref{defalpha}) directly as a feedback of $X_1$. 
\item Secondly, it is interesting to express the Lyapunov functional $V_1$ (defined by \eqref{defV1}) as a function of $X_1$.
\end{itemize}
For more details on how to invert the Artstein transform and how to use it in practice, we refer the readers to \cite{breschprieurtrelat}. Here, we develop only what is required to perform our stabilization analysis.

We have to solve the fixed point implicit equality \eqref{def_alpha}. For every function $f$ defined on $\R$ and locally integrable, we define 
$$
(T_Df)(t) = K_1(D) \int_{\max(t-D,D)}^t e^{(t-D-s)A_1}B_1 f(s)\, ds.
$$
It follows from \eqref{def_alpha} that $\alpha(t) = K_1(D)X_1(t)+(T_D\alpha)(t)$, for every $t\geq D$. A purely formal computation yields that 
$$
\alpha(t) = \sum_{j=0}^{+\infty} (T_D^j K_1(D)X_1)(t) .
$$
The convergence of the series is not obvious and is proved in the following lemma.

\begin{lemma}\label{lem_alpha}
We have
\begin{equation}\label{alphaexpanded}
\alpha(t) = \left\{ \begin{array}{ll}
0 & \textrm{if}\ t<D,\\
\displaystyle\sum_{j=0}^{+\infty} (T_D^j K_1(D)X_1)(t) & \textrm{if}\ t\geq D,
\end{array}\right.
\end{equation}
and the series is convergent, whatever the value of the delay $D\geq 0$ may be.
\end{lemma}

Note that the value of the feedback $\alpha$ at time $t$,
\begin{multline*}
\alpha(t) = K_1(D)X_1(t) \\+ K_1(D) \int_{\max(t-D,D)}^t e^{(t-D-s)A_1}B_1K_1(D) X_1(s)\, ds \\
+ K_1(D) \int_{\max(t-D,D)}^t e^{(t-D-s)A_1}B_1 K_1(D)\\ \int_{\max(s-D,D)}^s e^{(s-D-\tau)A_1}B_1K_1(D) X_1(\tau)\, d\tau \, ds \\
\qquad +\cdots
\end{multline*}
depends on the past values of $X_1$ over the time interval $(D,t)$.
Since the feedback is retarded with the delay $D$, the term $\alpha(t-D)$ appearing at the right-hand side of \eqref{systfini} only depends on the values of $X_1(s)$ with $0<s<t-D$, as desired.

\begin{proof}
We define the functions ${\varphi_D}_j$ iteratively by
\begin{equation}\label{defvarphij}
\begin{split}
{\varphi_D}_1(t,\tau)&=1,\\
{\varphi_D}_{j+1}(t,\tau)&=\int_{\max(\tau,t-D)}^{\min(t,\tau+jD)}{\varphi_D}_j(s,\tau)\, ds,\qquad \forall j\in\N^*, 
\end{split}
\end{equation}
for every $t\geq \tau$, and by ${\varphi_D}_{j}(t,\tau)= 0$ if $t<\tau$ and $ j\in\N$. 

Let us prove by induction that
\begin{multline}\label{induc1}
\left\vert (T_D^jK_1(D)X_1)(t) \right\vert \\
\leq \Vert B_1\Vert_{\R^{n+1}}^j \Vert K_1(D)\Vert_{\R^{n+1}}^{j+1} 
\\\times \int_{\max(t-jD,D)}^t {\varphi_D}_j(t,\tau) e^{(t-jD-\tau)\Vert A_1\Vert } \Vert X_1(\tau) \Vert_{\R^{n+1}} \, d\tau ,
\end{multline}
for every $j\in\N^*$.
This is clearly true for $j=1$, since
\begin{equation*}
\begin{array}{l}
\left\vert (T_DK_1(D)X_1)(t) \right\vert 
\\
= \left\vert K_1(D) \int_{\max(t-D,D)}^t e^{(t-D-s)A_1}B_1K_1(D)X_1(s)\, ds \right\vert \\
\leq \Vert B_1\Vert_{\R^{n+1}} \Vert K_1(D)\Vert^2_{\R^{n+1}} \\
\quad \times \int_{\max(t-D,D)}^t e^{(t-D-s) \Vert A_1\Vert} \Vert X_1(s)\Vert_{\R^{n+1}}\, ds .
\end{array}
\end{equation*}
Assume that this is true for an integer $j\in\N^*$, and let us derive the estimate for $j+1$.
Since
$$\begin{array}{l}
(T_D^{j+1}K_1(D)X_1)(t)
\\ = K_1(D)\int_{\max(t-D,D)}^t e^{(t-D-s)A_1}B_1  (T_D^jK_1(D)X_1)(s)    \, ds,
\end{array}$$
we get
\begin{equation*}
\begin{split}
& \left\vert (T_D^{j+1}K_1(D)X_1)(t) \right\vert 
\leq \Vert B_1\Vert_{\R^{n+1}}^{j+1} \Vert K_1(D)\Vert_{\R^{n+1}}^{j+2}   \\
&\;
\times \int_{\max(t-D,D)}^t e^{(t-D-s)\Vert A_1\Vert} \\
&\;\times
 \int_{\max(s-jD,D)}^s {\varphi_D}_j(s,\tau) e^{(s-jD-\tau)\Vert A_1\Vert } \Vert X_1(\tau) \Vert_{\R^{n+1}} \, d\tau  \, ds ,
\end{split}
\end{equation*}
and, from the Fubini theorem, noting that $(\tau,s)$ is such that 
$$
\left\{\begin{array}{l}\max(s-jD,D) \leq \tau \leq s \, , \\ \max(t-D,D) \leq s \leq t,\end{array}\right.
$$ 
if and only if
$$
\left\{\begin{array}{l}\max(t-(j+1)D,D) \leq \tau \leq t \, , \\  \max(\tau ,t-D)  \leq s \leq \min(t,\tau + jD ) ,\end{array}\right.
$$
we get the estimate
\begin{equation*}
\begin{split}
& \left\vert (T_D^{j+1}K_1(D)X_1)(t) \right\vert \leq \Vert B_1\Vert_{\R^{n+1}}^{j+1} \Vert K_1(D)\Vert_{\R^{n+1}}^{j+2}   \\
&\qquad
\times \int_{\max(t-(j+1)D,D)}^t \left(\int_{\max(\tau,t-D)}^{\min(t,\tau+jD)}{\varphi_D}_j(s,\tau)\, ds\right)\\
&\qquad \times \ e^{(t-(j+1)D-\tau)\Vert A_1\Vert}\Vert X_1(\tau) \Vert_{\R^{n+1}} \, d\tau ,
\end{split}
\end{equation*}
and the desired estimate for $j+1$ follows by definition of ${\varphi_D}_{j+1}$.

Now, we claim that
\begin{equation}\label{estimvarphij}
0\leq {\varphi_D}_{j}(t,\tau)\leq \frac{(t-\tau)^{j-1}}{(j-1)!},
\end{equation}
for every $j\in\N^*$. 
Indeed, nonnegativity is obvious and the right-hand side estimate easily follows from the fact that ${\varphi_D}_{j+1}(t,\tau)\leq \int_\tau^t{\varphi_D}_j(s,\tau)\, ds$ and from a simple iteration argument. 

Finally, from \eqref{induc1} and \eqref{estimvarphij}, we infer that
\begin{equation*}
\begin{array}{l}
\left\vert (T_D^jK_1(D)X_1)(t) \right\vert  \\
\leq \Vert B_1\Vert_{\R^{n+1}}^{j} \Vert K_1(D)\Vert_{\R^{n+1}}^{j+1}    \\
\qquad
\times \int_{\max(t-jD,D)}^t   \frac{(t-\tau)^{j-1}}{(j-1)!}   e^{(t-jD-\tau)\Vert A_1\Vert}\Vert X_1(\tau) \Vert_{\R^{n+1}} \, d\tau \\
 \leq  \Vert B_1\Vert_{\R^{n+1}}^{j} \Vert K_1(D)\Vert_{\R^{n+1}}^{j+1} \frac{(t-D)^j}{(j-1)!} \max_{D\leq s\leq t}\Vert X_1(s)\Vert_{\R^{n+1}},
\end{array}
\end{equation*}
whence the convergence of the series in \eqref{alphaexpanded}.
\end{proof}

\begin{remark}\label{remark4}
It is also interesting to express $Z_1$ in function of $X_1$, that is, to invert the equality
\begin{equation}\label{X1enfonctiondeZ1}
Z_1(t) \!\!= \!\!X_1(t) +\! \int_{(t-D,t)\cap(D,+\infty)} e^{(t-s-D)A_1}B_1 K_1(D) Z_1(s) \, ds
\end{equation}
coming from \eqref{transfo_Artstein} and \eqref{defalpha}. 
Although it is technical and not directly useful to derive the exponential stability of $Z_1$, it will however allow us to express the Lyapunov functional $V_1$ defined by \eqref{defV1}.
Note that
\begin{equation}\label{ref_intervalle}
(t-D,t)\cap(D,+\infty) = \left\{\begin{array}{ll}
\emptyset & \textrm{if}\ t<D, \\
(D,t) & \textrm{if}\ D<t<2D, \\
(t-D,t) & \textrm{if}\ 2D<t.
\end{array}\right.
\end{equation}
In particular if $t<D$ then $Z_1(t)=X_1(t)$.
We have the following result.

\begin{lemma}\label{lemtechnique}
For every $t\in\R$, there holds
\begin{equation}\label{X1Z1}
X_1(t) = Z_1(t) - \int_{(t-D,t)\cap(D,+\infty)} f(t-s) X_1(s)\, ds,
\end{equation}
where $f$ is defined as the unique solution of the fixed point equation
$$
f(r) =  f_0(r)  + (\tilde T_D f)(r),
$$
with $f_0(r)=e^{(r-D)A_1} B_1K_1(D)$ and
$$
(\tilde T_D f)(r) = \int_0^{r} e^{(r-\tau-D)A_1} B_1K_1(D) f(\tau)\, d\tau.
$$
Moreover, we have
\begin{equation*}
\begin{array}{rcl}
f(r) &\!\!=&\!\!\sum_{j=0}^{+\infty} (\tilde T_D^j f_0)(r) \\
& \!\!=& \!\!e^{(r-D)A_1} B_1K_1(D)  \\
&&\!\!\!+ \int_0^{r} e^{(r-\tau-D)A_1} B_1K_1(D) e^{(\tau-D)A_1} B_1K_1(D)   \, d\tau \\
&&\!\!\!+ 
\int_0^{r} e^{(r-\tau-D)A_1} B_1K_1(D)\\
&&\!\!\!\times \int_0^{\tau} e^{(\tau-s-D)A_1} B_1K_1(D) e^{(s-D)A_1} B_1K_1(D)   \, ds   \, d\tau \\
&&\!\!\!+ \cdots
\end{array}
\end{equation*}
and the series is convergent, whatever the value of the delay $D\geq 0$ may be.
\end{lemma}

The proof of this lemma is done in Section \ref{sec_proof_lemtechnique}.

With this expression and using (\ref{X1enfonctiondeZ1}) in Remark \ref{remark4}, the feedback $\alpha$ can be as well written as
\begin{equation*}
\begin{array}{rcl}
\alpha(t)&=& \chi_{(D,+\infty)}(t) K_1(D)Z_1(t) \\
& = &\chi_{(D,+\infty)}(t) K_1(D) X_1(t) 
\\&&+ K_1(D)\int_{(t-D,t)\cap(D,+\infty)} f(t-s) X_1(s)\, ds ,
\end{array}
\end{equation*}
and we recover of course the expression \eqref{alphaexpanded} derived in Lemma \ref{lem_alpha}.
\end{remark}

Plugging this feedback into the control system \eqref{systfini} yields, for $t>D$, the closed-loop system
\begin{equation}\label{closed-loop_system_complique}
\begin{split}
X'_1(t) &=A_1X_1(t) + B_1\alpha(t-D) \\
&= A_1X_1(t) + B_1K_1(D)X_1(t-D) 
\\ &\quad +B_1 K_1(D) \int_{(t-D,t)\cap(D,+\infty)} f(t-s) X_1(s)\, ds  ,
\end{split}
\end{equation}
which is, as said above, exponentially stable. Moreover, the Lyapunov function $V_1$, which is exponentially decreasing according to Remark \ref{rem_Lyap1}, can be written as
\begin{equation*}
\begin{array}{c}V_1(t) = \frac{1}{2} \left(X_1(t)+ \int_{I_t(D)} f(t-s) X_1(s)\, ds\right)^\top P(D)\\
\times \left(X_1(t)+ \int_{I_t(D)} f(t-s) X_1(s)\, ds\right) .
\end{array}\end{equation*}
with $I_t(D) = (t-D,t)\cap(D,+\infty)$.
We stress once again that the above feedback and Lyapunov functional stabilize the system whatever the value of the delay may be.

\subsection{Exponential stability of the entire system in closed-loop}
In order to prove that the feedback $\alpha$ designed above stabilizes the entire system \eqref{sys-dim-infinie}, we have to take into account the rest of the system (consisting of modes that are naturally stable). What has to be checked is whether the delay control part might destabilize this infinite-dimensional part or not.

Let $(u_D(\cdot),w(\cdot))$ denote a solution of \eqref{newzero} in which we choose the control $\alpha$ in the feedback form designed previously, such that $u_D(0)=0$ and $w(0)=0$.
Here, we make a slight abuse of notation, since $w(t)$ designates the solution $w(t,\cdot)\in H^2(0,L)\cap H^1_0(0,L)$ satisfying
\begin{equation}\label{eq444}
\begin{split}
& u_D' = \alpha,\quad w' = Aw+a\alpha_D+b\alpha ,\\
& u_D(0)=0, \quad w(0,\cdot)=0.
\end{split}
\end{equation}
Let $M(D)$ be a positive real number such that
\begin{multline}\label{defM}
M(D)\ > \  \Vert b \Vert_{L^2(0,L)}^2 \Vert K_1(D)\Vert_{\R^{n+1}}^2 \\
+ \max\left(2 \Vert a\Vert^2_{L^2(0,L)},  \frac{\max(\lambda_1,\ldots,\lambda_n) }{\lambda_{\min}(P(D))} \right)
\\\times \max\left(1,D e^{2D\Vert A_1\Vert} \Vert B_1\Vert_{\R^{n+1}}^2 \Vert K_1(D)\Vert_{\R^{n+1}}^2\right)  ,
\end{multline}
where 
$$
\Vert K_1(D)\Vert_{\R^{n+1}}^2 = \sum_{j=0}^{n} k_j(D)^2 ,
\qquad
\Vert B_1\Vert_{\R^{n+1}}^2 = 1+\sum_{j=1}^n b_j^2 ,
$$
$\Vert A_1\Vert$ is the usual matrix norm induced from the Euclidean norm of $\R^{n+1}$, and $\lambda_{\min}(P(D))>0$ is the smallest eigenvalue of the symmetric positive definite matrix $P(D)$. The precise value of $M(D)$ is not important however. What is important in what follows is that $M(D)>0$ is large enough.

We set
\begin{equation}\label{defVD}
\begin{array}{rcl}
V_D(t) &= &M(D)\,V_1(t) + M(D)\, \int_{(t-D,t)\cap(D,+\infty)} V_1(s)\, ds \\
&&- \frac{1}{2}\langle w(t),Aw(t)\rangle_{L^2(0,L)} \\
&=& \frac{M(D)}{2}Z_1(t)^\top P(D)Z_1(t) 
\\&&+ \frac{M(D)}{2} \int_{(t-D,t)\cap(D,+\infty)} Z_1(s)^\top P(D)Z_1(s)\, ds \\
&&
 - \frac{1}{2}\sum_{j=1}^{+\infty} \lambda_jw_j(t)^2.
\end{array}
\end{equation}
We are going to prove that $V_D(t)$ is positive and decreases exponentially to $0$. This Lyapunov functional consists of three terms. The two first terms stand for the unstable finite-dimensional part of the system. As we will see, the integral term is instrumental in order to tackle the delayed terms. The third term stands for the infinite-dimensional part of the system. In this infinite sum actually all modes are involved, in particular those that are unstable. Then the two first terms of \eqref{defVD}, weighted with $M(D)>0$, can be seen as corrective terms and this weight $M(D)>0$ is chosen large enough so that $V_D(t)$ be indeed positive.
More precisely,
\begin{equation}\label{12:55}
- \frac{1}{2}\sum_{j=1}^{+\infty} \lambda_jw_j(t)^2 = - \frac{1}{2}\sum_{j=1}^n \lambda_jw_j(t)^2 - \frac{1}{2}\sum_{j=n+1}^\infty \lambda_jw_j(t)^2,
\end{equation}
where $\lambda_j\geq 0$ for every $j\in\{1,\ldots,n\}$ and $\lambda_j\leq-\eta<0$ for every $j>n$ (see \eqref{refeta}). 
Therefore the second term of \eqref{12:55} is positive and the first term, which is nonpositive, is actually compensated by the first term of $V_D(t)$ since $M(D)$ is large enough, as proved in the following more precise lemma.

\begin{lemma} \label{lemma2.4}
There exists $C_2(D)>0$ such that
\begin{equation}\label{equiv2}
V_D(t) \geq C_2(D) \left( u_D(t)^2 + \Vert w(t)\Vert_{H^1_0(0,L)}^2 \right), 
\end{equation}
for every $t\geq 0$. 
\end{lemma}

\begin{proof}
First of all, by definition of $\lambda_{\min}(P(D))$, one has
\begin{equation}\label{lambdamin1}\begin{array}{c}
\frac{M(D)}{2}Z_1(t)^\top P(D)Z_1(t) \qquad\qquad\qquad\qquad\qquad\\
+ \frac{M(D)}{2} \int_{t-D}^t Z_1(s)^\top P(D)Z_1(s)\, ds \qquad\qquad\qquad\\\qquad\qquad\qquad
\geq \ M(D) \frac{\lambda_{\min}(P(D))}{2} \left( \Vert Z_1(t)\Vert_{\R^{n+1}}^2 \right.\\
\qquad\qquad\qquad\qquad\left.+ \int_{t-D}^t \Vert Z_1(s)\Vert_{\R^{n+1}}^2\, ds \right) ,
\end{array}\end{equation}
for every $t\geq 0$.
Besides, recall that, from \eqref{X1enfonctiondeZ1}, one has
$$\begin{array}{rcl}
X_1(t)& \!\!\!\!=& \!\!\!\!Z_1(t) \\
&&\!\!\!\!- \int_{(t-D,t)\cap(D,+\infty)} e^{(t-s-D)A_1}B_1 K_1(D) Z_1(s) \, ds,
\end{array}$$
and therefore, using the Cauchy-Schwarz inequality and the inequality $(a+b)^2\leq 2a^2+2b^2$, it follows that
\begin{equation}\label{est3}
\Vert X_1(t)\Vert_{\R^{n+1}}^2 \!\!\leq \!\!C_3(D)\left( \!\Vert Z_1(t)\Vert_{\R^{n+1}}^2 \!\!+\!\!\int_{t-D}^t \!\Vert Z_1(s)\Vert_{\R^{n+1}}^2\, ds\! \right)\!\!,
\end{equation}
with 
$$
C_3(D) = \max\left(2,2D e^{2D\Vert A_1\Vert} \Vert B_1\Vert_{\R^{n+1}}^2 \Vert K_1(D)\Vert_{\R^{n+1}}^2\right).
$$
We then infer from \eqref{lambdamin1} and \eqref{est3} that
\begin{multline}\label{lambdamin2}
\frac{M(D)}{2}Z_1(t)^\top P(D)Z_1(t)
\\ + \frac{M(D)}{2} \int_{t-D}^t Z_1(s)^\top P(D)Z_1(s)\, ds \\
\geq \ M(D) \frac{\lambda_{\min}(P(D))}{2C_3(D)} \Vert X_1(t)\Vert_{\R^{n+1}}^2 ,
\end{multline}
for every $t\geq 0$.

Using \eqref{12:55} and the definition of $X_1$ in (\ref{eq:def:A1}), we have
\begin{multline}\label{13:00}
- \frac{1}{2}\sum_{j=1}^{+\infty} \lambda_jw_j(t)^2 \geq - \frac{1}{2}\sum_{j=n+1}^\infty \lambda_jw_j(t)^2 \\-  \frac{1}{2} \max_{1\leq j\leq n} ( \lambda_j ) \Vert X_1(t)\Vert_{\R^{n+1}}^2 ,
\end{multline}
and therefore, using \eqref{lambdamin2}, we get
\begin{multline}\nonumber
V_D(t) \geq \left( M(D) \frac{\lambda_{\min}(P(D))}{2C_3(D)} \right.
- \left. \frac{1}{2} \max_{1\leq j\leq n} ( \lambda_j ) \right) \Vert X_1(t)\Vert_{\R^{n+1}}^2 \\ - \frac{1}{2}\sum_{j=n+1}^\infty \lambda_jw_j(t)^2 ,
\end{multline}
for every $t\geq 0$.
By definition of $M(D)$ (see \eqref{defM}), one has $M(D) \frac{\lambda_{\min}(P(D))}{2C_3(D)} -  \frac{1}{2} \max_{1\leq j\leq n} ( \lambda_j ) >0$ and hence there exists $C_4(D)>0$ such that
\begin{equation}\label{13:47}
V_D(t) \geq C_4(D) \left(  \Vert X_1(t)\Vert_{\R^{n+1}}^2  - \frac{1}{2}\sum_{j=n+1}^\infty \lambda_jw_j(t)^2 \right) .
\end{equation}

Using the series expansion $w(t,\cdot)=\sum_{i=1}^{+\infty} w_i(t)e_i(\cdot)$, we have
$$\Vert w(t)\Vert_{H^1_0(0,L)}^2=\sum_{(i,j)\in(\N^*)^2}w_i(t)w_j(t)\int_0^L e_i'(x) e_j'(x)\, dx.$$
By definition, one has $e_n''+ce_n=\lambda_ne_n$ and $e_n(0)=e_n(L)=0$, for every $n\in\N^*$. Integrating by parts and using the orthonormality property, we get
$$
\int_0^L e_i'(x) e_j'(x)\, dx =\int_0^L c(x)e_i(x)e_j(x) \, dx - \lambda_j\delta_{ij}  , 
$$
with $\delta_{ij}=1$ whenever $i=j$ and $\delta_{ij}=0$ otherwise, and thus, for all $t\geq 0$,
\begin{equation}\label{cp:1}
\Vert w(t)\Vert_{H^1_0(0,L)}^2=\int_0^L c(x)w(t,x)^2\,dx - \sum_{j=1}^\infty \lambda_j w_j(t)^2.
\end{equation}
Since $c\in L^\infty(0,L)$, it follows that
\begin{equation*}
\begin{split}
& \Vert w(t)\Vert_{H^1_0(0,L)}^2 \\
&\leq \Vert c\Vert_{L^\infty(0,L)}\ \Vert w(t)\Vert_{L^2(0,L)}^2 - \sum_{j=1}^n \lambda_j w_j(t)^2  - \sum_{j=n+1}^\infty \lambda_j w_j(t)^2 \\
&\leq \Vert c\Vert_{L^\infty(0,L)}\sum_{j=1}^\infty  w_j(t)^2  - \sum_{j=n+1}^\infty \lambda_j w_j(t)^2 \\
&\leq \Vert c\Vert_{L^\infty(0,L)}\Vert X_1(t)\Vert_{\R^{n+1}}^2  - \sum_{j=n+1}^\infty (\lambda_j- \Vert c\Vert_{L^\infty(0,L)}) w_j(t)^2 \\
\end{split}
\end{equation*}
and since $\lambda_j\rightarrow -\infty$ as $j$ tends to $+\infty$, there exists $C_5>0$ such that
$$
\Vert w(t)\Vert_{H^1_0(0,L)}^2 \leq - C_5 \left(  \Vert X_1(t)\Vert_{\R^{n+1}}^2  - \frac{1}{2}\sum_{j=n+1}^\infty \lambda_jw_j(t)^2 \right).
$$
Then \eqref{equiv2} follows from \eqref{13:47}.
\end{proof}

Using \eqref{ref_intervalle}, note that if $t<D$ then the integral term of \eqref{defVD} is equal to $0$ and $Z_1(t)=X_1(t)$, and hence
\begin{equation*}
V_D(t) = \frac{M(D)}{2}X_1(t)^\top P(D)X_1(t)  - \frac{1}{2}\sum_{j=1}^{+\infty} \lambda_jw_j(t)^2 ,
\end{equation*}
for every $t<D$. This remark leads to the following lemma.

\begin{lemma} \label{lemma2.5}
There exists $C_6(D)>0$ such that
\begin{equation}\label{equiv3}
V_D(t) \leq C_6(D) (u_D(t)^2+ \Vert w(t)\Vert_{H^1_0(0,L)}^2 ), 
\end{equation}
for every $t < D$.
\end{lemma}

\begin{proof}
Using \eqref{cp:1}, one has
$$ \begin{array}{rcl}
- \sum_{j=1}^{+\infty} \lambda_jw_j(t)^2
&\leq& \Vert w(t)\Vert_{H^1_0(0,L)}^2 
\\&&\quad + \Vert c\Vert_{L^\infty(0,L)} \ \Vert w(t)\Vert_{L^2(0,L)}^2 
\\&\leq&  C_8(D)\Vert w(t)\Vert_{H^1_0(0,L)}^2,\end{array}
$$
and then the lemma follows from the fact that $\Vert w(t)\Vert_{L^2(0,L)}^2 \leq L \Vert w(t)\Vert_{H^1_0(0,L)}^2$, obtained by the Poincar\'e inequality.
\end{proof}

\begin{lemma}\label{lem5}
The functional $V_D$ decreases exponentially to $0$.
\end{lemma}

\begin{proof}
Let us compute $V_D'(t)$ for $t>2D$ and state a differential inequality satisfied by $V_D$. 
First of all, it follows from \eqref{poleshifting} (in Corollary \ref{corkalman}) that
$$
\frac{d}{dt} \frac{M(D)}{2}Z_1(t)^\top P(D)Z_1(t) = -M(D) \Vert Z_1(t)\Vert_{\R^{n+1}}^2,
$$
and thus
\begin{equation*}
\begin{split}
\frac{d}{dt} \frac{M(D)}{2}\int_{t-D}^t Z_1(s)^\top P(D)Z_1(s)\, ds 
\\= -M(D) \int_{t-D}^t \Vert Z_1(s)\Vert_{\R^{n+1}}^2 \, ds .
\end{split}
\end{equation*}
Then, using \eqref{eq444}, \eqref{defVD} and the fact that $A$ is selfadjoint, we get
\begin{equation}\label{diffV1}\begin{array}{rcl}
V_D'(t) &=& -M(D) \Vert Z_1(t)\Vert_{\R^{n+1}}^2 \\
&& - M(D) \int_{t-D}^t \Vert Z_1(s)\Vert_{\R^{n+1}}^2 \, ds \\
 &&-  \Vert Aw(t)\Vert_{L^2(0,L)}^2 - \langle Aw(t),a\rangle_{L^2(0,L)}u_D(t) \\&&- \langle Aw(t),b\rangle_{L^2(0,L)}K_1(D)Z_1(t) ,
\end{array}\end{equation}
for every $t>2D$. From the Young inequality, we derive the estimates
\begin{multline}\label{est1}
\left\vert\langle Aw(t),a\rangle_{L^2(0,L)}u_D(t)\right\vert\\ \leq \frac{1}{4}\Vert Aw(t)\Vert^2_{L^2(0,L)}+ \Vert a\Vert^2_{L^2(0,L)} \Vert X_1(t)\Vert^2_{\R^{n+1}} ,
\end{multline}
and
\begin{multline}\label{est2}
\left\vert\langle Aw(t),b\rangle_{L^2(0,L)}K_1(D)Z_1(t)\right\vert\\ \leq \frac{1}{4}\Vert Aw(t)\Vert^2_{L^2(0,L)} + \Vert b \Vert_{L^2(0,L)}^2 \Vert K_1(D)\Vert_{\R^{n+1}}^2 \Vert Z_1(t)\Vert^2_{\R^{n+1}}.
\end{multline}
With the estimates \eqref{est1}, \eqref{est2} and \eqref{est3}, we infer from \eqref{est3} and from \eqref{diffV1} that
$$\begin{array}{rcl}
V_D'(t) &\leq& - \left( M(D)-\Vert b \Vert_{L^2(0,L)}^2 \Vert K_1(D)\Vert_{\R^{n+1}}^2 \right.
\\ &&\left.\qquad - \Vert a\Vert^2_{L^2(0,L)} C_3(D) \right) \Vert Z_1(t)\Vert_{\R^{n+1}}^2  \\
&&- \left( M(D) - \Vert a\Vert^2_{L^2(0,L)} C_3(D) \right) \int_{t-D}^t \Vert Z_1(s)\Vert_{\R^{n+1}}^2 \, ds  \\&&- \frac{1}{2} \Vert Aw(t)\Vert_{L^2(0,L)}^2 .
\end{array}$$
From \eqref{defM}, the real number $M(D)$ has been chose large enough so that 
$$
M(D)-\Vert b \Vert_{L^2(0,L)}^2 \Vert K_1(D)\Vert_{\R^{n+1}}^2  - \Vert a\Vert^2_{L^2(0,L)} C_3(D) >0 $$
and
$$
M(D) - \Vert a\Vert^2_{L^2(0,L)} C_3(D) >0 .
$$
Therefore, there exists $C_7(D)>0$ such that
\begin{equation}\label{14:14}\begin{array}{rcl}
V_D'(t) &\leq &-C_7(D) \left( \Vert Z_1(t)\Vert_{\R^{n+1}}^2 + \int_{t-D}^t \Vert Z_1(s)\Vert_{\R^{n+1}}^2 \, ds    \right) 
\\ &&- \frac{1}{2} \Vert Aw(t)\Vert_{L^2(0,L)}^2 .
\end{array}\end{equation}

Let us provide an estimate of $\Vert Aw(t)\Vert_{L^2(0,L)}^2$. Since $-\lambda_j\leq\lambda_j^2$ for any $j$ large enough, it follows that there exists $C_8>0$ such that
\begin{multline*}
-\frac{1}{2} \langle w(t),Aw(t)\rangle_{L^2(0,L)} \\
= -\frac{1}{2}\sum_{j=1}^{n} \lambda_j w_j(t)^2 -\frac{1}{2}\sum_{j=n}^{+\infty} \lambda_j w_j(t)^2 
\leq -\frac{1}{2}\sum_{j=n}^{+\infty} \lambda_j w_j(t)^2 \\
\leq \frac{1}{2C_8} \sum_{j=1}^{+\infty} \lambda_j^2 w_j(t)^2 
\leq \frac{1}{2C_8} \Vert Aw\Vert_{L^2(0,L)}^2 .
\end{multline*}
Hence it follows from \eqref{14:14} that 
$$\begin{array}{rcl}
V_D'(t) &\leq& -C_7(D) \left( \Vert Z_1(t)\Vert_{\R^{n+1}}^2 + \int_{t-D}^t \Vert Z_1(s)\Vert_{\R^{n+1}}^2 \, ds    \right)  \\
&&- \frac{C_8}{2}  \langle w(t),Aw(t)\rangle_{L^2(0,L)}  .
\end{array}$$
Finally, using \eqref{lambdamin1}, there exists $C_9(D)>0$ such that
$$
V_D'(t) \leq -C_9(D) V_D(t),
$$
for every $t>2D$. Therefore $V_D(t)$ decreases exponentially to $0$.
\end{proof}

From Lemma \ref{lem5}, $V_D(t)$ decreases exponentially to $0$. It follows from Lemmas \ref{lemma2.4} and \ref{lemma2.5} that there exists $C_{10}(D)>0$ and $\mu>0$ such that 
$$\begin{array}{rcl}
u_D(t) ^2 + \Vert w(t)\Vert_{H^1_0(0,L)}^2 &\leq &
C_{10}(D) e^{-\mu t} \\
&&\times (u_D(0) ^2 + \Vert w(0)\Vert_{H^1_0(0,L)}^2 ) ,
\end{array}$$
for every $t\geq 0$. Using \eqref{defw}, Theorem \ref{thm1} follows.

\section{Numerical simulation}\label{secnum}
In this section, we illustrate Theorem \ref{thm1} with an example and a numerical simulation. We take $c(x)=0.5$, for all $x\in(0,L)$, $L=2\pi$ and $D=1$. It is easily checked that, with a null boundary control, there is only one eigenvalue that is positive and thus there is only one mode of \eqref{eqcont0} that is unstable. Using a simple pole-shifting controller on the finite-dimensional linear control system resulting of the unstable part of \eqref{eq:def:A1} (with $(-0.5,-1)$ as poles for the closed-loop system), we compute (with \texttt{Matlab}) a stabilizing delay input for the infinite-dimensional system \eqref{eqcont0}. The overall numerical procedure to compute the controller is based on the discretization of the explicit form of the Artstein transformation for the finite-dimensional unstable part of \eqref{eqcont0} (as done in \cite{breschprieurtrelat} for finite dimensional control system with input delay). Then we discretize the reaction-diffusion equation \eqref{eqcont0} using the first $6$ modes, when closing the loop with this delay controller. We take as initial condition $y^0(x) = x(L-x)$.

The time evolution of the obtained solution is given on Figure \ref{fig:1} and the delayed boundary controller $u_D$ is given on Figure \ref{fig:2}. It can be checked on Figure \ref{fig:1} that, as expected, the solution converges to equilibrium.

\begin{figure}[ht]
	\centering
	\includegraphics[width = 0.9\columnwidth]{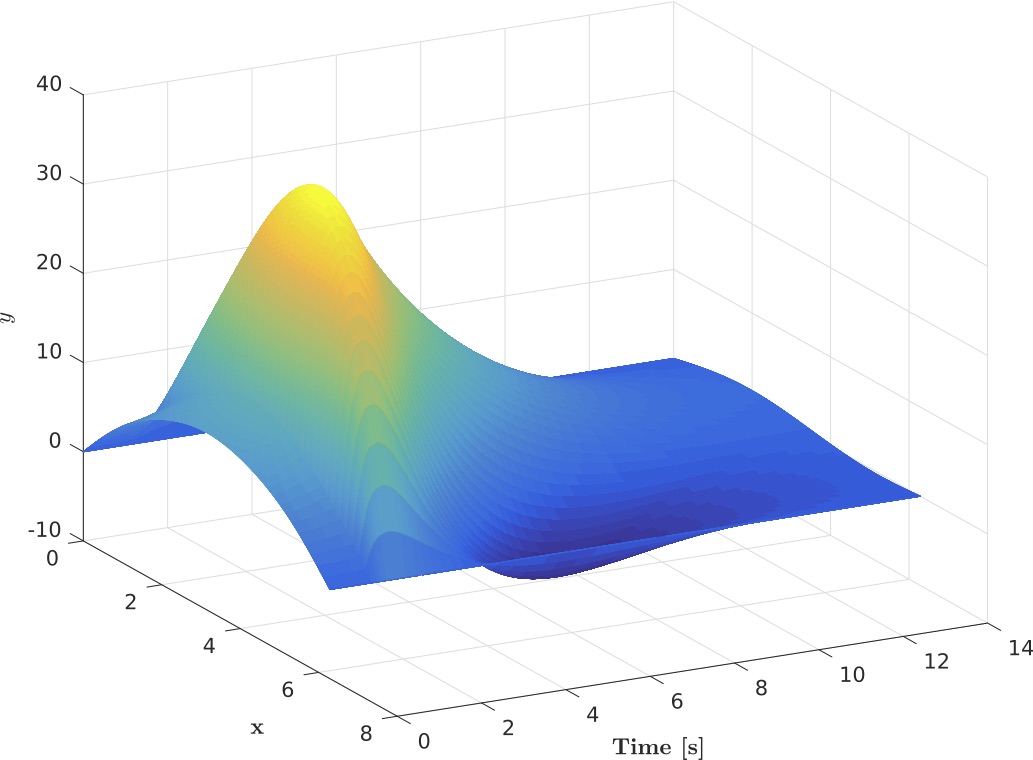} 
	\caption{Time-evolution of the solution of \eqref{eqcont0} with delay boundary controller}
	\label{fig:1}
\end{figure}

\begin{figure}[ht]
	\centering
	\includegraphics[width = 0.7\columnwidth]{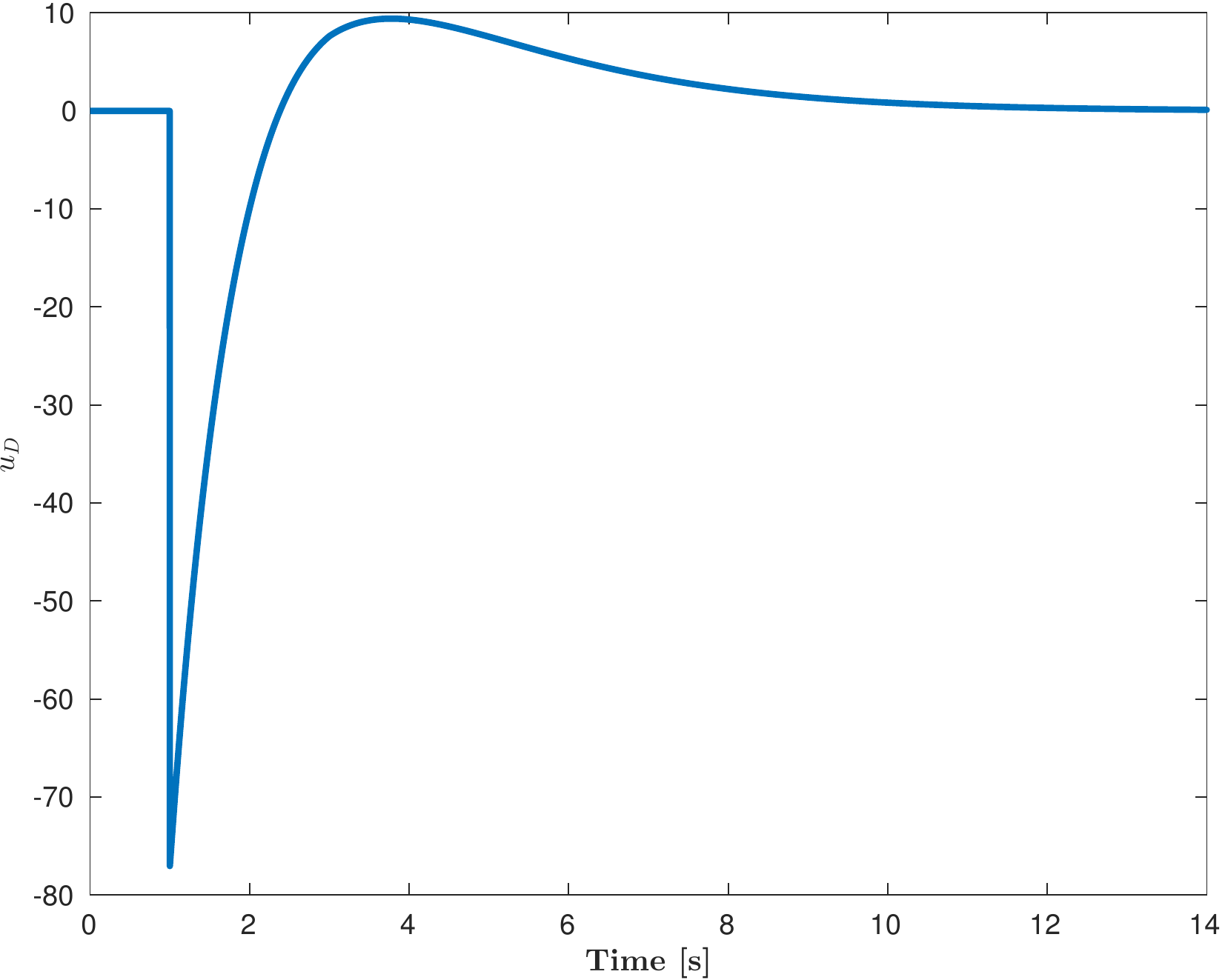} 
	\caption{Delay control $u_D$ for \eqref{eqcont0}}
	\label{fig:2}
\end{figure}


\section{Proof of Lemma \ref{lemtechnique}}\label{sec_proof_lemtechnique}
Let us search the kernel $\Phi_D$ such that
$$
X_1(t) = Z_1(t) - \int_{-\infty}^t \Phi_D(t,s)X_1(s)\, ds,
$$
postulating that $\Phi_D(t,s)=0$ whenever $s>t$.
Using \eqref{X1enfonctiondeZ1} we must have
\begin{equation*}
\begin{split}
& X_1(t) + \int_{-\infty}^t\Phi_D(t,s)X_1(s)\, ds \\
& = X_1(t) + \int_{(t-D,t)\cap(D,+\infty)} e^{(t-s-D)A_1} B_1K_1(D)
\\
&\quad \times \left( X_1(s) + \int_{-\infty}^s \Phi_D(s,\tau)X_1(\tau) \right) d\tau .
\end{split}
\end{equation*}
We have already noted that if $t<D$ then $Z_1(t)=X_1(t)$, and hence in that case $\Phi_D(t,s)=0$. Hence in what follows we assume that $t>D$. Using the Fubini theorem, we get
\begin{multline*}
\int_{-\infty}^t\Phi_D(t,s)X_1(s)\, ds
\\
=  \int_{\max(t-D,D)}^t e^{(t-s-D)A_1} B_1K_1(D) X_1(s) \, ds \\
+  \int_{-\infty}^t \int_{\max(t-D,D,s)}^t e^{(t-\tau-D)A_1} B_1K_1(D) \Phi_D(\tau,s)\, d\tau\ X_1(s)\, ds .
\end{multline*}
Since we would like this equality to hold true for every $X_1$, there must hold
\begin{equation}\label{volterra}
\begin{split}
\Phi_D(t,s) =\ & e^{(t-s-D)A_1} B_1K_1(D) \chi_{(\max(t-D,D),t)}(s) \\
& + \int_{\max(t-D,D,s)}^t e^{(t-\tau-D)A_1} B_1K_1(D) \Phi_D(\tau,s)\, d\tau \\
=\ & e^{(t-s-D)A_1} B_1K_1(D) \chi_{(\max(t-D,D),t)}(s) \\
& + \int_{\max(t-D,D,s)}^{t-s} e^{(t-\tau-D)A_1} B_1K_1(D) \Phi_D(\tau,s)\, d\tau \\
\end{split}
\end{equation}
Let us now solve the implicit equation \eqref{volterra}.

\medskip

First of all, if $D<t<2D$ then $\max(t-D,D)=D$ and \eqref{volterra} yields
$$\begin{array}{rcl}
\Phi_D(t,s) &= &  e^{(t-s-D)A_1} B_1K_1(D) \chi_{(D,t)}(s) 
 \\&&+  \int_{\max(s,D)}^t e^{(t-\tau-D)A_1}B_1K_1(D) \Phi_D(\tau,s)\, d\tau .
\end{array}$$
There are two subcases.

If $s<D$ or if $s>t$ then clearly $\Phi_D(t,s)=0$ is a solution.

If $D<s<t$ then
\begin{equation*}
\begin{split}
\Phi_D(t,s) &=   e^{(t-s-D)A_1} B_1K_1(D) \chi_{(D,t)}(s) 
 \\&\quad +  \int_s^t e^{(t-\tau-D)A_1} B_1K_1(D) \Phi_D(\tau,s)\, d\tau \\
&=  e^{(t-s-D)A_1} B_1K_1(D) \chi_{(D,t)}(s) \\
 &\quad + \int_0^{t-s} e^{(t-s-\tau-D)A_1} B_1K_1(D) \Phi_D(\tau+s,s)\, d\tau
\end{split}
\end{equation*}
and then setting $r=t-s$ (note that $0<r<2D$) we search $\Phi_D(t,s)=f(r)$ with
\begin{equation*}
\begin{split}
f(r) &=  e^{(r-D)A_1} B_1K_1(D)  + \int_0^{r} e^{(r-\tau-D)A_1} B_1K_1(D) f(\tau)\, d\tau\\
&= f_0(r) + (\tilde T_D f)(r)
\end{split}
\end{equation*}
with $f_0(r) = e^{(r-D)A_1} B_1K_1(D)$ and
$(\tilde T_D f)(r) = \int_0^{r} e^{(r-\tau-D)A_1} B_1K_1(D) f(\tau)\, d\tau$.
Formally, we get $f(r)=0$ whenever $r<0$ and $r>2D$, and
\begin{equation*}
\begin{split}
f(r) &= \sum_{j=0}^{+\infty} (\tilde T_D^j f_0)(r) \\
& = e^{(r-D)A_1} B_1K_1(D)  \\
&\qquad + \int_0^{r} e^{(r-\tau-D)A_1} B_1K_1(D) e^{(\tau-D)A_1} B_1K_1(D)   \, d\tau \\
&\qquad + 
\int_0^{r} e^{(r-\tau-D)A_1} B_1K_1(D) \\
&\qquad \times \int_0^{\tau} e^{(\tau-s-D)A_1} B_1K_1(D) e^{(s-D)A_1} B_1K_1(D)   \, ds   \, d\tau \\
&\qquad + \cdots
\end{split}
\end{equation*}
for every $r\in(0,2D)$.
The convergence of the series follows from the estimate
$$
\vert (\tilde T_D^j f_0)(r)\vert \leq \Vert B_1\Vert_{\R^{n+1}}^{j+1}\Vert K_1(D)\Vert_{\R^{n+1}}^{j+1} \frac{r^j}{j!} e^{(r-jD)\Vert A_1\Vert},
$$
which is immediate to establish by induction.

\medskip

If $t>2D$ then $\max(t-D,D)=t-D$ and \eqref{volterra} yields
\begin{equation*}
\begin{array}{rcl}\Phi_D(t,s) &= & e^{(t-s-D)A_1} B_1K_1(D) \chi_{(t-D,t)}(s) 
 \\&&+ \int_{\max(s,t-D)}^t e^{(t-\tau-D)A_1} B_1K_1(D) \Phi_D(\tau,s)\, d\tau 
\end{array}\end{equation*}
There are two subcases.

If $s<t-D$ or if $s>t$ then clearly $\Phi_D(t,s)=0$ is a solution.

If $t-D<s<t$ then
\begin{equation*}
\begin{array}{rcl}\Phi_D(t,s) &\!\!\!=&\!\!\!  e^{(t-s-D)A_1} B_1K_1(D) \chi_{(t-D,t)}(s) 
 \\&&\!\!\!+ \int_{s}^t e^{(t-\tau-D)A_1} B_1K_1(D) \Phi_D(\tau,s)\, d\tau \\
&\!\!\!=&\!\!\!  e^{(t-s-D)A_1} B_1K_1(D) \chi_{(t-D,t)}(s) 
 \\&&\!\!\!+ \int_0^{t-s} e^{(t-s-\tau-D)A_1} B_1K_1(D) \Phi_D(\tau+s,s)\, d\tau ,
\end{array}
\end{equation*}
and then setting $r=t-s$ (note that, then, $0<r<D$), similarly as above, we search $\Phi_D(t,s)=f(r)$ with
$f(r) = f_0(r) + (\tilde T_D f)(r)$ for every $r\in(0,D)$.
Formally, we get $f(r)=0$ whenever $r<0$ and $r>D$, and
$f(r) = \sum_{j=0}^{+\infty} (\bar T_D^j f_0)(r)$
for every $r\in(0,D)$. The convergence is established as previously.

\section{Conclusion}\label{sec:con}
For a reaction-diffusion equation with delay boundary control, a new constructive design method has been suggested. It is based on an explicit form of the classical Artstein transformation for the finite-dimensional unstable part of the delay system. By an appropriate Lyapunov function, it is shown that the designed boundary delay control stabilizes the entire reaction-diffusion partial differential equation. 
A numerical simulation shows how effective is this approach. 

This work lets some questions open. First, by noting that the studied system is equivalent to a scalar parabolic equation coupled with a scalar transport equation, it is natural to see if it is possible to adapt this design method to a system composed of several parabolic PDEs coupled with a hyperbolic system, coupled at the boundary (or inside by internal terms) and controlled by means of a delay controller.  Finally a degenerate reaction-diffusion system system has been studied in \cite{CrepeauPrieur08} for an approximate controllability problem. It is thus natural to investigate the stabilization problem of this PDE by means of a boundary delay control.

\bibliographystyle{plain}

\bibliography{cp}

\end{document}